\documentclass[a4paper,10pt]{article}

\usepackage{amsmath,amsfonts}

\usepackage[T1]{fontenc}

\newtheorem{Theorem}{\quad Theorem}[section]

\newtheorem{Lemma}[Theorem]{\quad Lemma}
\newtheorem{Proposition}[Theorem]{\quad Proposition}

\newcommand{\be} {\begin{equation}}
\newcommand{\ee} {\end{equation}}

\title{$ \sup \times \inf $ inequality on manifolds of dimension 5.}

\author{Samy Skander Bahoura \footnote{Adresses e-mail: samybahoura@gmail.com}}

\date{ \small Department of Mathematics, Pierre and Marie Curie University, 75005, Paris, France.}

\begin{document}

\maketitle

\begin{abstract}

We give an inequality of type $ \sup \times \inf $ in dimension 5 for a Yamabe type equation.

\smallskip


\end{abstract}

\section{Introduction and Main Result}

We consider a Riemannian manifold $ (M,g) $ of dimension $ n\geq 2 $ and we denote by $ \Delta=-\nabla^i(\nabla_i) $ the Laplace-Beltrami operator.

\smallskip

The Yamabe type equation (also called Schrodinger equation) in dimension $ n\geq 3 $ is:

$$ \Delta u+hu= n(n-2)u^{N-1}, \,\, u >0, \,\, N=\frac{2n}{n-2}.$$

with $ h $ a smooth function. When $ h=\frac{n-2}{4(n-1)} R_g $ with $ R_g $ the scalar curvature of $ (M,g)$, the equation is the Yamabe equation. If the number $ n(n-2) $ is replaced by a function $ V $, the equation is called, the prescribed scalar curvature equation.

\smallskip

About the existence result for the Yamabe equation, see [1].

\smallskip

In this paper, we are interested by inequalities of type $ \sup \times \inf $ relative to Yamabe type equation.

\smallskip

These type of estimates were studied for the prescribed scalar curvature equation and Yamabe equation in [1-20]. 

\smallskip

In dimension 2 and on domains of $ {\mathbb R}^2 $, Brezis-Li-Shafrir, see [8], proved  that $ \sup + \inf $ is bounded when the prescibed scalar curvature are Lipschitz, in [11], Chen-Lin extend this result to H\"older prescribed scalar curvatures.

\smallskip

In dimensions $ n\geq 3 $, there are many results about $ \sup \times \inf $ inequalities, see, [2-9], [11-12], [17-19].

\smallskip

Here, we are on a Riemannian manifold $ (M,g) $ of dimension $ n= 5 $, and we consider a Yamabe type equation:

\be \Delta u+hu=15 u^{7/3},\,\, u >0, \ee

with $ h $ a smooth function, $ ||h||_{C^2(M)} \leq h_0, h_0 >0 $. (Here: $n(n-2)=15 $).

When $ M $ is compact without boundary, this equation, where studied by many authors. See [1,13,15,17,19].

\smallskip

An important result in dimensions $ n=3,4 $ by Li-Zhang in [18], they proved the boundedness of $ \sup \times \inf $ for the Yamabe equation.

\smallskip

Our main result is:
 
\begin{Theorem} For all compact subset $ K $ of $ M $, there exists a positive constant $ c=c(K,M, h_0, g)>0 $ such that:

$$ (\sup_K u)^{1/7} \times \inf_M u \leq c. $$

for all $ u >0 $ solution of (1).

\end{Theorem}

In [19], Li-Zhu, gave an example of solution of a Yamabe type equation on domain of $ {\mathbb R}^5 $ with $ \sup \times \inf \to +\infty $. Here , we obtain a weaker inequality.

\smallskip

In [3], we considered the Yamabe equation. In this paper we take $ h\not =\frac{3}{16} R_g $, with $ R_g $ the scalar curvature. Here we consider more general equation, of Yamabe type.

\section{Proof of the Theorem:}

\bigskip

\underbar {\bf Blow-up analysis:}

\bigskip

Let $ u $ a function on $ M $ such that,

$$ \Delta u + h u =15{u}^{7/3}, \,\, u>0. $$

We argue by contradiction and we suppose that $ \sup \times \inf $ is not bounded.

We suppose:

$ \forall \,\, c,R >0 \,\, \exists \,\, u_{c,R} $ solution of $ (E) $ such that:

$$ R^{3} (\sup_{B(x_0,R)} u_{c,R})^{1/7} \times \inf_{M} u_{c,R} \geq c, \qquad (H) $$

\bigskip

\begin{Proposition} 

\smallskip

There are, $ (y_i)_i $, $ y_i \to x_0 $ and $ (l_i)_i, (L_i)_i $, $ l_i \to 0 $, $ L_i \to +\infty $, such that if we set, $ v_i(y)=\dfrac{u_io\exp_{y_i}(y)}{u_i(y_i)} $, we have:

$$ 0 < v_i(y) \leq  \beta_i \leq 2^{3/2}, \,\, \beta_i \to 1. $$

$$  v_i(y)  \to \left ( \dfrac{1}{1+{|y|^2}} \right )^{3/2}, \,\, {\rm uniform\'ement \,\, sur \,\, tout \,\, compact \,\, de } \,\, {\mathbb R}^5 . $$

$$ l_i^{3/2} [u_i(y_i)]^{1/7} \times \inf_M u_i \to +\infty. $$

\end{Proposition}

{\bf Proof proposition 2.1:}

\bigskip

We use $ (H) $, we suppse that, there are $ R_i>0, R_i \to 0 $ et $ c_i \to +\infty $, such that,

$$ {R_i}^{3}(\sup_{B(x_0,R_i)} u_i)^{1/7} \inf_{M} u_i \geq c_i \to +\infty, $$

Let, $ x_i \in  { B(x_0,R_i)} $, such that $ \sup_{B(x_0,R_i)} u_i=u_i(x_i) $ and $ s_i(x)=[R_i-d(x,x_i)]^{3/2} u_i(x), x\in B(x_i, R_i) $. Then, $ x_i \to x_0 $.

\bigskip

We have,

$$ \max_{B(x_i,R_i)} s_i(x)=s_i(y_i) \geq s_i(x_i)={R_i}^{3/2} u_i(x_i)\geq \sqrt {c_i}  \to + \infty. $$ 

\bigskip

We set :

$$ l_i=R_i-d(y_i,x_i),\,\, \bar u_i(y)= u_i o \exp_{y_i}(y),\,\,  v_i(z)=\dfrac{u_i o \exp_{y_i}\left ( z/[u_i(y_i)]^{2/3} \right ) } {u_i(y_i)}. $$

Clearly, $ y_i \to x_0 $. We obtain:

$$ L_i= \dfrac{l_i}{(c_i)^{1/6}} [u_i(y_i)]^{2/3}=\dfrac{[s_i(y_i)]^{2/3}}{c_i^{1/6}}\geq \dfrac{c_i^{1/3}}{c_i^{1/6}}=c_i^{1/6}\to +\infty. $$

\bigskip

If $ |z|\leq L_i $, then $ y=\exp_{y_i}[z/ [u_i(y_i)]^{2/3}] \in B(y_i,\delta_i l_i) $ with $ \delta_i=\dfrac{1}{(c_i)^{1/6}} $ and $ d(y,y_i) < R_i-d(y_i,x_i) $, thus, $ d(y, x_i) < R_i $ and thus, $ s_i(y)\leq s_i(y_i) $, we can write,

$$ u_i(y) [R_i-d(y,x_i)]^{3/2} \leq u_i(y_i) (l_i)^{3/2}. $$

but, $ d(y,y_i) \leq \delta_i l_i $, $ R_i >l_i$ et $ R_i-d(y, x_i) \geq R_i-d(y_i,x_i)-\delta_i l_i>l_i-\delta_i l_i=l_i(1-\delta_i) $, we obtain,

$$ 0 < v_i(z)=\dfrac{u_i(y)}{u_i(y_i)} \leq \left [ \dfrac{l_i}{l_i(1-\delta_i)} \right ]^{3/2}\leq 2^{3/2} . $$

We set, $ \beta_i=\left ( \dfrac{1}{1-\delta_i} \right )^{3/2} $, clearly $ \beta_i \to 1 $.

\bigskip

The function $ v_i $ satisfies:

$$ -g^{jk}[\exp_{y_i}(y)]\partial_{jk} v_i-\partial_k \left [ g^{jk}\sqrt { |g| } \right ][\exp_{y_i}(y)]\partial_j v_i+ \dfrac{ h [\exp_{y_i}(y)]}{[u_i(y_i)]^{4/(n-2)}} v_i=n(n-2){v_i}^{N-1}, $$

We use elliptic estimates to have a subsequence, $ ( v_i)_i $ which converge uniformly on compact subsets of the euclidean space to a function $ v $ solution of, 

$$ \Delta v=15v^{7/3}, \,\, v(0)=1,\,\, 0 \leq v\leq 1\leq 2^{3/2}, $$

By the maximum principle, $ v>0 $ on $ {\mathbb R}^5 $. By a result of Caffarelli-Gidas-Spruck (see [10]), we have: $ v(y)=\left ( \dfrac{1}{1+{|y|^2}} \right )^{3/2} $.

\bigskip

\underbar {\bf Polar geodesic coordinates and moving-plane method}

\bigskip

We set:

$$ w_i(t,\theta)=e^{3t/2}\bar u_i(e^t,\theta) = e^{3t/2}u_io\exp_{y_i}(e^t\theta), \,\, {\rm and} \,\, a(y_i,t,\theta)=\log J(y_i,e^t,\theta). $$ 

\begin{Lemma}

\smallskip

La fonction $ w_i $ est solution de:

$$  -\partial_{tt} w_i-\partial_t a \partial_t w_i+\Delta_{\theta}w_i+c w_i=15 w_i^{7/3}, $$
 
with,

 $$ c = c(y_i,t,\theta)=\left ( \dfrac{3}{2} \right )^2+ \dfrac{3}{2} \partial_t a + h e^{2t}, $$ 

\end{Lemma}

{\bf Proof of the Lemma}. See [3].

\bigskip

We write, $ \partial_t a=\dfrac{ \partial_t b_1}{b_1} $, $ b_1(y_i,t,\theta)=J(y_i,e^t,\theta)>0 $,

\bigskip

and,

$$ -\dfrac{1}{\sqrt {b_1}}\partial_{tt} (\sqrt { b_1} w_i)+\Delta_{\theta}w_i+[c(t)+ b_1^{-1/2} b_2(t,\theta)]w_i=15{w_i}^{7/3}, $$

o\`u, $ b_2(t,\theta)=\partial_{tt} (\sqrt {b_1})=\dfrac{1}{2 \sqrt { b_1}}\partial_{tt}b_1-\dfrac{1}{4(b_1)^{3/2}}(\partial_t b_1)^2 .$ 

\bigskip

We set,

$$ \tilde w_i=\sqrt {b_1} w_i, $$

\begin{Lemma}

The function $ \tilde w_i $ is solutions of:

$$ -\partial_{tt} \tilde w_i+\Delta_{\theta} (\tilde w_i)+2\nabla_{\theta}(\tilde  w_i) .\nabla_{\theta} \log (\sqrt {b_1})+(c+b_1^{-1/2} b_2-c_2) \tilde w_i= $$

$$ = 15\left (\dfrac{1}{b_1} \right )^{3/2} {\tilde w_i}^{7/3}, $$

with $ c_2 $ a function.

\end{Lemma}

\bigskip

{\bf Proof of the Lemma}. See [3] with $ c_2=[\dfrac{1}{\sqrt {b_1}} \Delta_{\theta}(\sqrt{b_1}) + |\nabla_{\theta} \log (\sqrt {b_1})|^2] . $

\bigskip

\underbar {\bf The moving-plane method:}

\bigskip

Let $ \xi_i $ a number, we suppose $ \xi_i \leq t $, we set $ t^{\xi_i}=2\xi_i-t $ and $ \tilde w_i^{\xi_i}(t,\theta)=\tilde w_i(t^{\xi_i},\theta) $.

\begin{Proposition}

We have:

$$ 1)\,\,\, \tilde w_i(\lambda_i,\theta)-\tilde w_i(\lambda_i+4,\theta) \geq \tilde k>0, \,\, \forall \,\, \theta \in {\mathbb S}_{4}. $$

 For all $ \beta >0 $, there is $ c_{\beta} >0 $ such that:

$$ 2) \,\,\, \dfrac{1}{c_{\beta}} e^{3t/2}\leq \tilde w_i(\lambda_i+t,\theta) \leq c_{\beta}e^{3t/2}, \,\, \forall \,\, t\leq \beta, \,\, \forall \,\, \theta \in {\mathbb S}_{4}. $$

\end{Proposition}

{\bf Proof of the proposition}. See [3].

\bigskip

We set:

$$ \bar Z_i=-\partial_{tt} (...)+\Delta_{\theta} (...)+2\nabla_{\theta}(...) .\nabla_{\theta} \log (\sqrt {b_1})+(c+b_1^{-1/2} b_2-c_2)(...) $$

{\bf Remark:} In the operator $ \bar Z_i $ we have: $ c+b_1^{-1/2}b_2-c_2 $ v\'erifie:

$$ c+b_1^{-1/2}b_2-c_2 \geq k'>0,\,\, {\rm pour }\,\, t<<0, $$

thus we can apply the Hopf maximum principle.

\bigskip

\underbar {\bf Goal:}

\bigskip

Like in [2,3], we want to prove that:

$$ \bar Z_i(\bar w_i^{\xi_i}-\bar w_i) \leq 0, \,\, {\rm si} \,\, \bar w_i^{\xi_i}-\bar w_i \leq 0. $$

with, $ \bar w_i $, a particular function obtained from $ w_i $.

We have:

\smallskip

\begin{Lemma}

$$ b_1(y_i,t,\theta)=1-\dfrac{1}{6} Ricci_{y_i}(\theta,\theta)e^{2t}+\ldots, $$

$$ h(e^t\theta)=h(y_i) + <\nabla h(y_i)|\theta > e^t+\dots . $$

\end{Lemma}.

and,

\begin{Proposition}

$$ \bar Z_i(\tilde w_i^{\xi_i}-\tilde w_i) \leq {b_1}^{-3/2}[(\tilde w_i^{\xi_i})^{7/3}- \tilde w_i^{7/3}]+  $$

$$ +C|e^{2t}-e^{2t^{\xi_i}}|\left [|\nabla_{\theta} {\tilde w_i}^{\xi_i}| + |\nabla_{\theta}^2(\tilde w_i^{\xi_i})|+ |Ricci_{y_i}|[\tilde w_i^{\xi_i}+\tilde w_i^{\xi_i})^{7/3}] + |h| \tilde w_i^{\xi_i} \right ] + C'|e^{3t^{\xi_i}}-e^{3t}|. $$

\end{Proposition}

{\bf Proof of the proposition}. See [3].

\bigskip

We have:

$$ a(y_i,t,\theta)=\log J(y_i,e^t,\theta)=\log b_1, |\partial_t b_1(t)|+|\partial_{tt} b_1(t)|+|\partial_{tt} a(t)|\leq C e^{2t}, $$

and,

$$ |\partial_{\theta_j} b_1|+|\partial_{\theta_j,\theta_k} b_1|+\partial_{t,\theta_j}b_1|+|\partial_{t,\theta_j,\theta_k} b_1|\leq C e^{2t}, $$

thus,

$$ |\partial_t b_1(t^{\xi_i})-\partial_t b_1(t)|\leq C'|e^{2t}-e^{2t^{\xi_i}}|,\,\, {\rm sur} \,\, ]-\infty, \log \epsilon_1]\times {\mathbb S}_{4},\forall \,\, x\in B(x_0,\epsilon_1) $$

locally,

$$ \Delta_{\theta}=-\dfrac{1}{\sqrt {|{\tilde g}^k(e^t,\theta)|}}\partial_{\theta^l}[{\tilde g}^{\theta^l \theta^j}(e^t,\theta)\sqrt { |{\tilde g}^k(e^t,\theta)|}\partial_{\theta^j}] . $$

Thus, in the chart$ ]0,\epsilon_1[\times U^k $, we have,

$$ A_i=\left [{ \left [ \dfrac{1}{\sqrt {|{\tilde g}^k|}}\partial_{\theta^l}({\tilde g}^{\theta^l \theta^j}\sqrt { |{\tilde g}^k|}\partial_{\theta^j}) \right ] }^{\xi_i}- \dfrac{1}{\sqrt {|{\tilde g}^k|}}\partial_{\theta^l}({\tilde g}^{\theta^l \theta^j}\sqrt { |{\tilde g}^k|}\partial_{\theta^j}) \right ](\tilde w_i^{\xi_i}) $$

thus, $ A_i=B_i+D_i $ with,

$$ B_i=\left [ {\tilde g}^{\theta^l \theta^j}(e^{t^{\xi_i}}, \theta)-{\tilde g}^{\theta^l \theta^j}(e^t,\theta) \right ] \partial_{\theta^l \theta^j}\tilde w_i^{\xi_i}(t,\theta), $$

and,

$$ D_i=\left [ \dfrac{1}{ \sqrt {| {\tilde g}^k|}(e^{t^{\xi_i}},\theta )  }           \partial_{\theta^l}[{\tilde g }^{\theta^l \theta^j}(e^{t^{\xi_i}},\theta)\sqrt {| {\tilde g}^k|}(e^{t^{\xi_i}},\theta)  ] -\dfrac{1}{ \sqrt {| {\tilde g}^k|}(e^t,\theta) } \partial_{\theta^l} [{\tilde g }^{\theta^l \theta^j}(e^t,\theta)\sqrt {| {\tilde g}^k|}(e^t,\theta) ] \right ] \partial_{\theta^j} \tilde w_i^{\xi_i}(t,\theta), $$

thus,

$$ A_i \leq C_k|e^{2t}-e^{2t^{\xi_i}}|\left [ |\nabla_{\theta} \tilde w_i^{\xi_i}| + |\nabla_{\theta}^2(\tilde w_i^{\xi_i})| \right ], $$

We take $ C=\max \{ C_i, 1 \leq i\leq q \} $.

\bigskip

We have

$$ c(y_i,t,\theta)=\left ( \dfrac{3}{2} \right )^2+ \dfrac{3}{2} \partial_t a + h e^{2t}, \qquad (\alpha_1) $$ 

$$ b_2(t,\theta)=\partial_{tt} (\sqrt {b_1})=\dfrac{1}{2 \sqrt { b_1}}\partial_{tt}b_1-\dfrac{1}{4(b_1)^{3/2}}(\partial_t b_1)^2 ,\qquad (\alpha_2) $$ 

$$ c_2=[\dfrac{1}{\sqrt {b_1}} \Delta_{\theta}(\sqrt{b_1}) + |\nabla_{\theta} \log (\sqrt {b_1})|^2], \qquad (\alpha_3) $$

then,

$$ \partial_{t}c(y_i,t,\theta)=\dfrac{(n-2)}{2}\partial_{tt}a+2e^{2t}h(e^t\theta)+e^{3t}<\nabla h(e^t\theta)|\theta >, $$

and then,

$$ |\partial_tc_2|+|\partial_t b_1|+|\partial_t b_2|+|\partial_t c|\leq K_1e^{2t}, $$

We know thta $ 0 < v_i(y) \leq 2^{3/2} $, by the elliptic estimates we obtain,

 $$ ||\nabla v_i||_{L^{\infty}(B(0,R)}+||\nabla^2 v_i||_{L^{\infty}(B(0,R)} \leq C(R,n). \qquad (***2) $$

We are in dimension $ 5 $. We consider the function $ \bar w_i(t,\theta)=\tilde w_i(t,\theta)-\dfrac{[u_i(y_i)]^{\alpha /3 } \min_M u_i}{2} e^{2t} $.

\bigskip

With $ t \leq t_i=-\dfrac{2 \alpha}{3} \log u_i(y_i) $, we have: 

$$ \bar w_i(t,\theta) = e^{2t}\left [\sqrt{b_1}(t,\theta) e^{-t/2}u_i o \exp_{y_i}(e^t\theta)-\dfrac{[u_i(y_i)]^{\alpha /3} \min_M u_i}{2} \right ] \geq $$

$$ \geq e^{2t} \dfrac{[u_i(y_i)]^{\alpha /3} \min_M u_i}{2}>0, $$

We set, $ \mu_i=\dfrac{[u_i(y_i)]^{\alpha /3 } \min_M u_i}{2} $.

\bigskip

Like in[2,3] we have: 

\begin{Lemma}

There is $ \nu <0 $ such that for $ \lambda \leq \nu $ :

$$ \bar w_i^{\lambda}(t,\theta)-\bar w_i(t,\theta) \leq 0, \,\, \forall \,\, (t,\theta) \in [\lambda,t_i] \times {\mathbb S}_{4}, $$

\end{Lemma}

{\bf Proof of the lemma}. See[2,3].

\bigskip

We set: $ \lambda_i=-\dfrac{2}{3} \log u_i(y_i) $, then, 

\begin{Lemma}

$$ \bar w_i(\lambda_i,\theta)-\bar w_i(\lambda_i+4,\theta) >0. $$

\end{Lemma}.

{\bf Proof of the lemma}. See [2,3].

\bigskip

Let, $ \xi_i=\sup \{ \lambda \leq \lambda_i+2, \bar w_i^{\xi_i}(t,\theta)-\bar w_i(t,\theta) \leq 0, \,\, \forall \,\, (t,\theta)\in [\xi_i,t_i]\times {\mathbb S}_{4} \} $.

\bigskip

$ \xi_i $ exists (see [2]), We obtain:

$$ \tilde w_i^{\xi_i}(t,\theta) + |\nabla_{\theta} \tilde  w_i^{\xi_i}(t,\theta)|+|\nabla_{\theta}^2  \tilde w_i^{\xi_i}(t,\theta)|\leq C(R),\,\,\, \forall \,\, (t,\theta) \in ]-\infty,\log R]\times {\mathbb S}_{4}, $$

We write:

$$ \bar Z_i(\bar w_i^{\xi_i}-\bar w_i)=\bar Z_i(\tilde w_i^{\xi_i}-\tilde w_i)-\mu_i \bar Z_i(e^{2t^{\xi_i}}-e^{2t}), $$

$$ -\bar Z_i(e^{2t^{\xi_i}}-e^{2t})=[4-\dfrac{9}{4}- \dfrac{3}{2}\partial_t a-he^{2t}+b_1^{-1/2} b_2-c_2](e^{2t^{\xi_i}}-e^{2t}) \leq c_1(e^{2t^{\xi_i}}-e^{2t}), $$

with $ c_1>0 $, car $ |\partial_t a|+|\partial_t b_1|+|\partial_{tt} b_1|+|\partial_{t,\theta_j} b_1|+|\partial_{t,\theta_j,\theta_k} b_1|\leq C'e^{2t}<1 $, for $ t $ very small.

\bigskip

We obtain on $ [\xi_i,t_i]\times {\mathbb S}_{4} $,

$$ \bar Z_i(\bar w_i^{\xi_i}-\bar w_i) \leq c_2[(\tilde w_i^{\xi_i})^{7/3}-\tilde w_i^{7/3}]+[\mu_i c_1- C'(R)](e^{2t^{\xi_i}}-e^{2t}) \leq 0, $$

\bigskip

Like in [2,3],  we use the Hopf maximum principle to conclude that:

$$ \sup_{\theta \in {\mathbb S}_{4}} \bar w_i^{\xi_i}(t_i,\theta) \geq \inf_{\theta \in {\mathbb S}_{4}} \bar w_i(t_i,\theta), $$

Thus,

$$  e^{3t_i/2} \min_{B(x_0,\epsilon_1)} u_i \leq \tilde c e^{3(\lambda_i-t_i)/2}, $$

Thus,

$$ [ u_i(y_i)]^{1-2\alpha } \min_{B(x_0,\epsilon_1)} u_i = e^{3(2t_i-\lambda_i)/2} \min_{B(x_0,\epsilon_1)} u_i \leq \tilde c, $$

which it is a contradiction for $ \alpha=\dfrac{3}{7} $.

\bigskip

Thus,

$$ \exists \,\, R>0, \,\,\exists \,\, c=c(M,g,R)>0 \,\,\, [ \sup_{B(x_0,R)} u_i ]^{1/7} \times \inf_{M} u_i \leq c \,\, \forall \,\, i $$

\end{document}